\documentclass{amsart}

\usepackage{amsfonts,amssymb,amsmath,amsthm}
\usepackage{url}
\usepackage{enumerate}

\usepackage{amsmath,amsfonts,amsthm,url,color,amssymb}
\usepackage{graphicx}

\newcommand{\fqn}{\mathbb{F}_{q^n}}
\newcommand{\F}{\mathbb{F}}

\newcommand{\oord}{\mathrm{Ord}}
\newcommand{\I}{\mathbb I}

\newtheorem{theorem}{Theorem}[section]
\newtheorem{lemma}[theorem]{Lemma}
\newtheorem{proposition}[theorem]{Proposition}
\newtheorem{corollary}[theorem]{Corollary}

\theoremstyle{definition}
\newtheorem{definition}[theorem]{Definition}
\newtheorem{example}[theorem]{Example}

\theoremstyle{remark}
\newtheorem{remark}[theorem]{Remark}

\numberwithin{equation}{section}

\begin{document}
\title{Normal points on Artin-Schreier curves over finite fields}


\author{Giorgos Kapetanakis}
\address{Department of Mathematics, University of Thessaly, 3rd km Old National Road Lamia-Athens, 35100, Lamia, Greece}
\email{kapetanakis@uth.gr}
\author{Lucas Reis}
\address{Departamento de Matem\'{a}tica,
Universidade Federal de Minas Gerais,
UFMG,
Belo Horizonte MG (Brazil),
 31270901}
\email{lucasreismat@mat.ufmg.br}



\date{\today}
\begin{abstract}
In 2022, S.D. Cohen and the two authors introduced and studied the concept of $(r, n)$-freeness on finite cyclic groups $G$ for suitable integers $r, n$, which is an arithmetic way of capturing elements of special forms that lie in the subgroups of $G$. Combining this machinery with some character sum techniques, they explored the existence of points $(x_0, y_0)$ on affine curves $y^n=f(x)$ defined over a finite field $\F$ whose coordinates are generators of the multiplicative cyclic group $\F^*$. 
In this paper we develop the natural additive counterpart of this work for finite fields. Namely, any finite extension $\mathbb E$ of a finite field $\F$ with $Q$ elements is a cyclic $\F[x]$-module induced by the Frobenius automorphism $\alpha\mapsto \alpha^{Q}$, and any generator of this module is said to be a normal element over $\F$. We introduce and study the concept of $(f, g)$-freeness on this module structure for suitable polynomials $f, g\in \F[x]$. As a main application of the machinery developed in this paper, we study the existence of $\F_{p^n}$-rational points in the Artin-Schreier curve $\mathfrak A_f : y^p-y=f(x)$ whose coordinates are normal over the prime field $\F_p$ and establish concrete results.
%
%
\end{abstract}

%
%
%
%

\maketitle

\section{Introduction}\label{sec:intro}
Let $q$ be a power of the prime $p$ and, for each positive integer $n$, let $\F_{q^n}$ be the finite field with $q^n$ elements. The field $\F_{q^n}$ has interesting structures related to the two basic field operations. Namely, the multiplicative group $\F_{q^n}^*=\F_{q^n}\setminus \{0\}$ is cyclic and any generator of such group is a {\em primitive element}. On the other hand, regarding the additive structure of $\F_{q^n}$, we can view $\F_{q^n}$ as an $\F_q[x]$-module induced by the Frobenius map $\alpha\mapsto \alpha^q$, namely
\[ \sum_{i=0}^{m}a_ix^i\circ \alpha:=\sum_{i=0}^ma_i\alpha^{q^i}. \]
It turns out that in this setting $\F_{q^n}$ is also cyclic and, regarding $\F_{q^n}$ as an $\F_q$-vector space, this means that there exists an $\F_q$-basis of the form $\{\beta, \beta^q, \ldots, \beta^{q^{n-1}}\}$. In this case, any such $\beta\in \F_q$ is a {\em normal element} over  $\F_q$. Both primitive and normal elements were extensively studied in the past decades, mainly motivated by theoretical problems, but also practical issues where such elements are employed. For instance, primitive elements are used in cryptographic applications such as the discrete logarithm problem (most notably, the Diffie-Hellman key exchange~\cite{dh}). Moreover, normal elements can be useful in generic situations where finite field arithmetic is performed; see~\cite{GAO} for an overview on normal elements, and their theoretical and practical aspects.

In 1987, Lenstra and Schoof~\cite{lenstraschoof87} proved that, for every prime power $q$ and every positive integer $n\ge 2$, there exists an element $\alpha\in \F_{q^n}$ that is simultaneously primitive and normal over $\F_q$. This is widely known as the {\em Primitive Normal Basis Theorem} (PNBT). A crucial tool in the proof of the PNBT is to provide expressions for the characteristic functions of normal and primitive elements in finite fields by means of additive and multiplicative characters, respectively. The latter is obtained through the concept of {\em freeness}, which is a convenient way to capture elements that can be written in special forms when one considers the aforementioned cyclic group and $\F_q[x]$-module structure of finite fields. For more details, see Definition 5.1 in~\cite{HMPT}. 
In the original proof by Lenstra and Schoof~\cite{lenstraschoof87}, after some algebraic computations and estimates on certain character sums, the PNBT is proved up to a finite number of pairs $(q, n)$. For these remaining pairs, the theorem is verified by direct computer search. A computer-free proof of the PNBT was later given by Cohen and Huczynska in 2003 in \cite{cohenhuczynska03}. 

In the past decade, questions related to the PNBT have been proposed and many results have been established. In this context, a recurrent object of study is the existence of pairs of elements $(\alpha, f(\alpha))\in \F_q\times \F_q, f\in \F_q(x)$ with special properties related to the multiplicative and additive structure of $\F_q$ (e.g., prescribed multiplicative order or $k$-normality\footnote{For the formal definition and basic properties of $k$-normal elements, see \cite{HMPT}.}).
Although the number of works in this line of research is extensive, we refer the interested reader, for example, to \cite{BCLT} or \cite{kapetanakisreis18} and the references therein.
Such pairs can be viewed as points $(x_0, y_0)$ on affine curves $y=f(x)$ whose coordinates $x_0, y_0$ have the aforementioned properties.  Motivated by the latter, in~\cite{cohenkapetanakisreis22} the authors explore the existence of points $(x_0, y_0)\in \F_q\times \F_q$ on affine curves $y^n=f(x)$ such that $x_0$ and $y_0$ are both primitive elements of $\F_q$. In order to obtain existence results on the latter they generalized the concept of freeness for generic multiplicative finite cyclic groups, culminating in a character sum expression for the set of elements in $\F_{q^n}^*$ with prescribed multiplicative order. 

In this paper paper we develop the natural additive counterpart of the concepts and results that are provided in \cite{cohenkapetanakisreis22}. Towards this end, we first present some background material in Section~\ref{sec:preparation}, then we introduce and study $(f,g)$-freeness for the additive structure of $\F_{q^n}$ in Sections~\ref{sec:freeness} and \ref{sec:freepolval}. Finally, in Section~\ref{sec:ascurves} we confine ourselves to extensions over the prime field $\F_p$ and employ the developed theory, in order to study the existence of $\F_q$-rational points in the Artin-Schreier curve $\mathfrak A_f : y^p-y=f(x)$ with $\F_p$-normal coordinates. In particular, our main result is the following theorem.

\begin{theorem}\label{thm:as}
Let $q=p^n$ be a prime power, where $n\ge 5$ and let $f\in \F_q[x]$ be a polynomial that is not of the form $ax^p+bx+c$ with $a, b, c\in \F_q$ satisfying $1<\deg(f)\le p+1$. Then there exists a point $(x_0, y_0)\in \F_q\times \F_q$ in the Artin-Schreier curve $\mathfrak A_f: y^p-y=f(x)$ such that both $x_0, y_0$ are normal over $\F_p$, provided that 
$(n,p) \neq (5,2), (5,5), (6,2), (6,3)$ or $(6,7)$.
\end{theorem}

\begin{remark} \label{remark1}
Theorem~\ref{thm:as} entails that the pairs $(n, p)=(5,2), (5,5), (6,2), (6,3)$ and $(6,7)$ are \emph{possible} exceptions, not necessarily genuine exceptions. However, a computer check reveals that the cases $(5, 2)$ and $(6, 2)$ are indeed genuine exceptions to Theorem~\ref{thm:as}: note that in this case we necessarily have $\deg(f)=3$. For the pair $(5, 2)$  we get  exactly $4$ exceptions while for the pair $(6, 2)$ we get thousands of exceptions. For the remaining pairs, the number of possible $f$'s is extremely large (about $p^{n(p+1)}$) and we were not able to check them fully, so we get no conclusion for them.
\end{remark}

\begin{remark}
In this work, we employ the so-called \emph{prime sieve}, see Theorem~\ref{thm:sieve}. In \cite{bailey19}, an improvement of this technique, called the \emph{modified prime sieve}, was introduced. However, the three \emph{possible} exceptions, $(n,p) = (5,5), (6,3)$ and $(6,7)$, see Theorem~\ref{thm:as} and Remark~\ref{remark1}, fail to pass the resulting condition even if the modified prime sieve is employed.
\end{remark}

\section{Preparation} \label{sec:preparation}

This section provides background material that will be used along the way. Throughout this paper, $q$ is a prime power and $\F_q$ is the finite field with $q$ elements, $\F_{q^n}$ is its extension of degree $n$ and $\overline\F_q$ is its algebraic closure.

\subsection{Some arithmetic functions over $\F_q[x]$}
We present some arithmetic functions defined over polynomials that are further used.

\begin{definition}
Let $f, g\in \F_q[x]$ be nonzero polynomials. 
\begin{enumerate}[(i)]
\item We set $f_{(g)}=\frac{f}{\gcd(f, g)}$.
\item $\Phi_q(f)$ denotes the Euler totient function for polynomials, i.e., $\Phi_q(f)=\# \left(\frac{\F_q[x]}{f\cdot \F_q[x]}\right)^{\times}$ is the number of invertible cosets modulo $f(x)$.
\item $W(f)$ stands for the number of monic squarefree divisors of $f$ in $\F_q[x]$.
\item $\mu_q(f)$ denotes the M\"obius function for polynomials over $\F_q$. More precisely, $\mu_q(f)=0$ if $f$ is not squarefree and $\mu_q(f)=(-1)^r$ if $f$ has $r\ge 0$ distinct irreducible monic divisors, defined over $\F_q$.
\item $|f| = \# \left(\frac{\F_q[x]}{f\cdot \F_q[x]}\right) = q^{\deg f}$ is the number of cosets of $\F_q[x]$ modulo $f(x)$.
\end{enumerate}
\end{definition}

It is well-known that if $\varphi=\Phi_q, W$ or $\mu_q$, then $\varphi(F\cdot G)=\varphi(F)\cdot \varphi(G)$ for all relatively prime polynomials $F, G\in \F_q[x]$, i.e., these functions are multiplicative.  We will need the following result.

\begin{lemma}\label{lem:arithfun}
For nonzero polynomials $f, g\in \F_q[x]$, we have that
\[ T(f, g):=\sum_{t\mid f}\frac{|\mu_q(t_{(g)}|}{\Phi_q(t_{(g)})}\cdot \Phi_q(t)= |\gcd(f, g)|\cdot W(\gcd(f, f_{(g)})). \]
\end{lemma}
\begin{proof}
We observe that the functions $F_g(f) := \frac{|\mu_q(f_{(g)}|}{\Phi_q(f_{(g)})}\cdot \Phi_q(f)$ and $G_g(f) := |\gcd(f, g)| \cdot W(\gcd(f, f_{(g)}))$ are both multiplicative on $f$. In particular the multiplicativity of $F_g(f)$ implies the multiplicativity of $T(f,g)$ on $f$. 
Thus, it suffices to consider the case when $f$ is a power of an irreducible polynomial. So, we assume that $f=h^\kappa$, where $h\in\F_q[x]$ is irreducible and $\kappa\geq 0$. Further, write $g=h^\lambda r$, where $\lambda\geq 0$ and $\gcd(r,h)=1$.

First, assume that $\kappa>\lambda$. Then
\begin{align*}
T(f,g) & = \sum_{t\mid f}\frac{|\mu_q(t_{(g)}|}{\Phi_q(t_{(g)})} \cdot \Phi_q(t) \\
 & = \sum_{i=0}^{\lambda} \frac{|\mu_q(1)|}{\Phi_q(1)} \cdot \Phi_q(h^i) + \sum_{i=\lambda+1}^\kappa \frac{|\mu_q(h^{i-\lambda})|}{\Phi_q(h^{i-\lambda})}  \cdot \Phi_q(h^i) \\
 & = |h^\lambda| + \frac{\Phi_q(h^{\lambda+1})}{\Phi_q(h)} = 2\cdot |h^\lambda| .
\end{align*}
Also, $\gcd(f,g) =h^\lambda$ and $\gcd(f, f_{(g)} ) = h^{\kappa-\lambda}$, hence $G_g(f) = 2\cdot |h^\lambda| = T (f,g)$.

Next, assume that $\kappa\leq\lambda$. In this case,
\[ T(f,g) = \sum_{t\mid f}\frac{|\mu_q(t_{(g)}|}{\Phi_q(t_{(g)})} \cdot \Phi_q(t) 
  = \sum_{i=0}^{\kappa} \frac{|\mu_q(1)|}{\Phi_q(1)} \cdot \Phi_q(h^i) = |h^\kappa| , \]
while, $\gcd(f,g) =h^\kappa$ and $\gcd(f, f_{(g)} ) = 1$, hence $G_g(f) = |h^\kappa| = T (f,g)$. The desired result follows.
\end{proof}

\subsection{The $\F_q$-order of an element in $\overline{\F}_q$}
We start with the following definition.
\begin{definition}\label{def:circ}
For a polynomial $f\in \F_q[x]$ and $\alpha\in \overline{\F}_q$ with $f(x)=\sum_{i=0}^ma_ix^i$, we set $f\circ \alpha=\sum_{i=0}^ma_i\alpha^{q^i}$.

\end{definition}

The existence of normal elements is known for any extension $\F_{q^n}$ of $\F_q$, see Theorem~2.35 of \cite{lidlniederreiter97}. 
The following well-known technical results are presented proofless. For more details, see Section~3.4 of \cite{lidlniederreiter97}.

\begin{lemma}\label{lem:q-associate}
For any $f, g\in \F_q[x]$ and any $\alpha\in \overline{\F}_q$, we have that $(f+g)\circ \alpha=f\circ \alpha+g\circ \alpha$ and $(f\cdot g)\circ \alpha=f\circ(g\circ \alpha)$.
\end{lemma}

\begin{lemma}\label{lem:generator}
Fix $n$ a positive integer and $\beta\in \F_{q^n}$ a normal element. Then any $\alpha\in \F_{q^n}$ is written uniquely as $f\circ \beta$ for some $f\in \F_q[x]$ of degree at most $n-1$.
\end{lemma}

If $\alpha\in \overline{\F}_{q}$, notice that $(x^n-1)\circ \alpha=0$ if and only if $\alpha^{q^n}-\alpha=0$, i.e., $\alpha\in \F_{q^n}$. In particular, from Lemma~\ref{lem:q-associate}, the set $\mathcal I_{\alpha}:=\{h\in \F_q[x]\,|\,h\circ \alpha=0\}$ is a nonzero ideal of $\F_q[x]$, hence is generated  by a unique monic polynomial in $\F_q[x]$. This polynomial,  denoted by $\oord(\alpha)$, is called the \emph{$\F_q$-order} of $\alpha$.

The following lemma relates the $\F_q$-order of $\alpha=h\circ \beta$ with the $\F_q$-order of $\beta$ in a natural way. For its proof, see Lemma 2.6 in~\cite{R20}.

\begin{lemma}\label{lem:gcd}
Let $\beta\in \overline{\F}_q$ and fix $g\in \F_q[x]$. If $\alpha=g\circ \beta$, then
\[ \oord(\alpha)=\frac{\oord(\beta)}{\gcd(\oord(\beta), g(x))} = \oord(\beta)_{(g)}. \]
\end{lemma}

It follows by the definition of $\oord(\alpha)$ that $\alpha\in \F_{q^n}$ if and only if $\oord(\alpha)\mid x^n-1$. From this observation and Lemmas~\ref{lem:generator} and~\ref{lem:gcd}, we readily obtain the following corollaries.

\begin{corollary}
An element $\beta\in \F_{q^n}$ is normal over $\F_q$ if and only if $\oord(\beta)=x^n-1$.
\end{corollary}

\begin{corollary}\label{cor:ord}
Let $n$ be a positive integer and $\alpha, \beta \in \fqn$, where $\beta$ is normal over $\F_q$. For a monic divisor $f\in \F_q[x]$ of $x^n-1$, we have that $\oord(\alpha)=f$ if and only if $\alpha=h\circ \beta$, where $h=\frac{x^n-1}{f}\cdot g$ and $g\in \F_q[x]$ is of degree smaller than $\deg(f)$ and $\gcd\left(g, f\right)=1$. In particular, for each monic divisor $f\in \F_q[x]$ of $x^n-1$, there exist $\Phi_q(f)$ elements in $\F_{q^n}$ with $\F_q$-order $f$.
\end{corollary}

\subsection{Additive characters}
Write $q=p^k$. An additive character of $\F_{q^n}$ is a homomorphism $\psi$ between the additive group $(\F_{q^n}, +)$ and the multiplicative group $\mathbb C^{\times}$ of invertible complex numbers. The \emph{canonical} additive character of $\F_{q^n}$, denoted by $\psi_1$, is the map $\alpha \mapsto \exp\left(\frac{\mathrm{Tr}(\alpha)}{p}\right)$, where $\mathrm{Tr}(\alpha)=\sum_{i=0}^{kn}\alpha^{p^i}\in \F_p$ is the absolute field trace and $\exp(z)=e^{2\pi i z}$ denotes the complex exponential function. The set $\widehat{\fqn}$ of additive characters of $\fqn$ forms a multiplicative abelian group whose elements are the characters $\psi_a: \F_{q^n}\to \mathbb C^{\times}, a\in \fqn$ with $\psi_a(x)=\psi_1(ax)$; the character $\psi_0$ maps all the elements of $\fqn$ to $1\in \mathbb C$ and is called {\em trivial}. Any other character in $\widehat{\fqn}$ is called {\em nontrivial}.

Similarly to the additive group $\F_{q^n}$, the set $\widehat{\fqn}$ has an $\F_q[x]$-module structure. For $f\in \F_q[x]$ and $\psi\in \widehat{\fqn}$, the map $x\mapsto \psi(f\circ x)$ defines another character of $\widehat{\fqn}$, which we denote by $f\circ \psi$. From the results of the previous subsection, we easily deduce that, for $\psi\in \widehat{\F_{q^n}}$, the set $\mathcal I_{\psi}:=\{h\in \F_q[x]\,|\, h\circ \psi=\psi_0\}$ is a nonzero ideal of $\F_q[x]$, hence it is generated  by a unique monic polynomial in $\F_q[x]$. This polynomial is the \emph{$\F_q$-order} of $\psi$ and it is denoted by $\oord(\psi)$.
%

It is clear that $\oord(\psi)$ is a divisor of $x^n-1$. Conversely, for each monic divisor $f\in \F_q[x]$ of $x^n-1$, there exist $\Phi_q(f)$ additive characters of $\fqn$ with $\F_q$-order $f$.

\begin{remark}\label{rem:trivial}
It is clear that the trivial character is the only additive character of $\F_{q^n}$ of $\F_q$-order 1.
\end{remark}

We conclude this section with an important auxiliary result on additive character sums, see Theorem~5.38 of \cite{lidlniederreiter97}.

\begin{theorem} \label{thm:weil}
Let $P\in \F_{q^n}[x]$ be a polynomial not of the form $r(x)^p-r(x)+\delta$ with $\delta\in \F_{q^n}$ 
and let $\psi$ be a nontrivial additive character of $\fqn$. Then
\[ \left|\sum_{c\in \fqn}\psi(P(c))\right|\le (\deg(P)-1)q^{n/2}. \]
\end{theorem}

The above implies that polynomials of the form $r(x)^p-r(x)+\delta$, with $\delta\in \F_{q^n}$, are special, so the following definition is essential.

\begin{definition} \label{def:singular}
Let $P\in \F_{q^n}[x]$. If there exists some $r\in\F_{q^n}[x]$ and $\delta\in\F_{q^n}$, such that $P(x) = r(x)^p-r(x)+\delta$, then $P$ is called \emph{singular}. Otherwise, it is called \emph{nonsingular}. 
\end{definition}

\section{Introducing $(f, g)$-freeness} \label{sec:freeness}

In this section we introduce the $(f,g)$-free elements. We stress that these elements are in fact the additive analogues to $(r,n)$-free elements, as they were introduced in \cite{cohenkapetanakisreis22}. First, we recall the (additive) concept of freeness. For a positive integer $n$ and a divisor $g\in \F_q[x]$ of $x^n-1$, an element $\alpha\in \F_{q^n}$ is \emph{$g$-free} if the equality $\alpha=h\circ \beta$ with $h$ a monic divisor of $g$ and $\beta\in \F_{q^n}$ implies $h(x)=1$ and $\alpha=\beta$.

\begin{definition}
Let $n$ be a positive integer and let $f, g\in \F_q[x]$ be such that $g(x)$ divides $x^n-1$ and $f(x)$ divides $\frac{x^n-1}{g(x)}$. An element $\alpha\in \F_{q^n}$ is \emph{$(f, g)$-free} if the following hold:
\begin{enumerate}[(i)]
\item $\oord(\alpha)$ divides $\frac{x^n-1}{g(x)}$, i.e., $\frac{x^n-1}{g(x)}\circ \alpha=0$;
\item $\alpha$ is $f$-free over the set of roots of the equation $\frac{x^n-1}{g(x)}\circ y=0$, i.e., if $\alpha=f_0\circ \beta$ with $f_0\in \F_q[x]$ a monic divisor of $f$ and $\frac{x^n-1}{g(x)}\circ \beta=0$, then $f_0(x)=1$ and $\alpha=\beta$.
\end{enumerate}
\end{definition}

The following is straightforward.

\begin{remark}\label{remark:freeness}
Let $n$ be a positive integer and let $f \in\F_q[x]$ be a divisor of $x^n-1$,
\begin{enumerate}[(i)]
\item as $(x^n-1)\circ \alpha=0$ for every $\alpha\in \F_{q^n}$, the $(f, 1)$-free elements of $\F_{q^n}$ are just the usual $f$-free elements;
\item for $f\mid x^n-1$, the $\left(\frac{x^n-1}{f}, f\right)$-free elements of $\F_{q^n}$ are exactly the elements $\alpha$ with $\oord(\alpha)=\frac{x^n-1}{f(x)}$. In particular, an element $\alpha \in \F_{q^n}$ is normal over $\F_q$ if and only if it is $(x^n-1, 1)$-free.
\end{enumerate}
\end{remark}
The following lemma characterizes the $(f, g)$-free elements based on their $\F_q$-orders.
\begin{lemma} \label{lem:fgfree}
  Let $g\mid x^n-1$ and $f\mid \frac{x^n-1}{g}$. Then some $\alpha\in \F_{q^n}$ is $(f,g)$-free if and only if $\alpha=g\circ\beta$ for some $\beta\in \F_{q^n}$ but $\alpha$ is not of the form $(gp)\circ\gamma$ with $\gamma\in \F_{q^n}$, for every irreducible factor $p\in \F_q[x]$ of $f$. In particular, $\alpha\in \F_{q^n}$ is $(f,g)$-free if and only if  
$\gcd\left(fg, \frac{x^n-1}{\oord(\alpha)}\right)=g$.  
  \end{lemma}
  \begin{proof}
  Clearly, the set $\{g\circ\beta : \beta\in \F_{q^n} \}$ describes the elements of $\F_{q^n}$ whose $\F_q$-order divides $\frac{x^n-1}{g}$, thus, the first statement follows directly by the definition of $(f,g)$-free elements.
  
%

 For the second statement, from Lemma~\ref{lem:generator}, $\alpha = h'\circ\delta'$, for some normal $\delta'\in\F_{q^n}$ and some $h'\in\F_q[x]$ of degree $\leq n-1$.  Also, if we set $h=\gcd(h',x^n-1)$, Lemma~\ref{lem:gcd} entails that $\oord(\alpha)=\frac{x^n-1}{h}$.
Now, from the first part of the proof and the definition of $(f, g)$-freeness, it follows that $\alpha$ is $(f, g)$-free if and only if $\frac{x^n-1}{h}$ divides $\frac{x^n-1}{g}$ but does not divide $\frac{x^n-1}{gp}$ for any irreducible factor $p\in \F_q[x]$ of $f$. In other words, $\alpha$ is $(f, g)$-free if and only if $g=hs$ where $\gcd(s, f)=1$. Since $\oord(\alpha)=\frac{x^n-1}{h}$, we have that   \[ \gcd\left(fg, \frac{x^n-1}{\oord(\alpha)}\right)=g\cdot \gcd(f, s), \] from where the result follows.
  \end{proof}
\begin{remark}\label{remark:squarefree}
Observe that $(f,g)$-freeness is equivalent to $(f',g)$-freeness, where $f'$ can be any polynomial in $\F_q[x]$ dividing $\frac{x^n-1}{g}$ that has exactly the same monic irreducible factors with $f$. In particular, we can replace $f$ by its squarefree part. This will be done without further mention.
\end{remark}

In the proceeding sections, we will need a convenient expression, using character sums, of the characteristic function of $(f,g)$-free elements of $\F_{q^n}$, i.e., of the function
\[ \I_{f,g}(\alpha) := \begin{cases}
  1,& \text{if } \alpha\text{ is } (f,g)\text{-free}, \\
  0,& \text{otherwise},
\end{cases} \]
where $\alpha\in\F_{q^n}$. Towards this end, we prove the following, which is the additive analogue to Proposition~3.6 of \cite{cohenkapetanakisreis22} and, in fact, the arguments we use are merely an adaption of the ones found in the proof of Proposition~3.6 of \cite{cohenkapetanakisreis22}, adjusted accordingly to the present context.
\begin{proposition}\label{prop:vinogradov}
Let $f,g\in\F_q[x]$ be such that $g\mid x^n-1$ and $f\mid \frac{x^n-1}{g(x)}$. Then, for every $\alpha\in\F_{q^n}$, we have that
\[ \I_{f,g}(\alpha) = \frac{\Phi_q(f)}{|fg|} \sum_{t\mid fg} \frac{\mu_q(t_{(g)})}{\Phi_q(t_{(g)})} \sum_{\oord(\psi) = t} \psi(\alpha) , \]
where in the outer sum, the polynomial $t$ is monic and polynomial division is over $\F_q$.
\end{proposition}
\begin{proof}
Take some $\alpha\in\F_{q^n}$. Let $p_1,\ldots ,p_n$ be the distinct monic irreducible factors of $f$
over $\F_q$. Lemma~\ref{lem:fgfree} implies that $\alpha$ is $(f,g)$-free if and only if $\alpha$ is of the form $g\circ\beta$ for some $\beta\in\F_{q^n}$, but not of the form $(gp_i)\circ\beta$ for any $1\leq i\leq n$ and $\beta\in\F_{q^n}$. It follows that if $I_h$ is the characteristic function for elements of the form $h\circ\beta$, where $h\in\F_q[x]$, such that $h\mid x^n-1$, and $\beta \in\F_{q^n}$, then
\begin{equation}\label{eq:I1} 
\I_{f,g}(\alpha) = I_g(\alpha) \prod_{i=1}^n (1-I_{gp_i}(\alpha)) . 
\end{equation}
Clearly $I_g(\alpha)I_{gp_i}(\alpha) = I_{gp_i}(\alpha)$, for every $\alpha\in\F_{q^n}$ and $i=1,\ldots,n$, hence, Eq.~\eqref{eq:I1} yields
\begin{equation}\label{eq:I2} 
\I_{f,g}(\alpha) = \sum_{d\mid f} \mu_q(d) I_{gd}(\alpha), 
\end{equation}
where the sum is over the monic divisors of $f$, defined over $\F_q$.
Regarding $I_h$, the orthogonality relations imply that, for every $h\in\F_q[x]$, such that $h\mid x^n-1$,
\[ 
I_h(\alpha) = \frac{1}{|h|} \sum_{\oord(\psi) \mid h} \psi(\alpha) = \frac{1}{|h|} \sum_{d\mid h}\sum_{\oord(\psi)=d}\psi(\alpha) , 
\]
for every $\alpha\in\F_{q^n}$. Now, Eq.~\eqref{eq:I2} becomes
\begin{align}
\I_{f,g}(\alpha) & = \frac{1}{|g|}\sum_{d\mid f}\sum_{e\mid gd} \frac{\mu_q(d)}{|d|} \sum_{\oord(\psi)=e}\psi(\alpha) = \frac{1}{|g|} \sum_{d\mid f}\sum_{e\mid gd}A(d)B_\alpha(e) \nonumber \\
 & = \frac{1}{|g|} \sum_{e\mid fg} \sum_{d\mid \frac{f}{e_{(g)}}} A(e_{(g)}d)B_\alpha(e) = \frac{1}{|g|} \sum_{e\mid fg} B_\alpha(e) \sum_{d\mid \frac{f}{e_{(g)}}} A(e_{(g)}d) , \label{eq:I3}
\end{align}
where, for each $h\in \F_q[X]$,  we have that $A(h) := \mu_q(h)/|h|$ and $B_\alpha(h) := \sum_{\oord(\psi)=h}\psi(\alpha)$. Regarding the inner sum in the last expression, we have that
\begin{multline*}
\sum_{d\mid \frac{f}{e_{(g)}}} A(e_{(g)}d) = \sum_{d\mid \frac{f}{e_{(g)}}} \frac{\mu_q(e_{(g)}d)}{|e_{(g)}d|} = \frac{\mu_q(e_{(g)})}{|e_{(g)}|} \sum_{\substack{d\mid f \\ \gcd(d,e_{(g)})=1}} \frac{\mu_q(d)}{|d|} \\ = \frac{\mu_q(e_{(g)})}{|e_{(g)}|} \cdot \frac{\Phi_q(f_{e,g})}{|f_{e,g}|} =  \frac{\mu_q(e_{(g)})}{\Phi_q(e_{(g)})} \cdot \frac{\Phi_q(e_{(g)}f_{e,g})}{|e_{(g)}f_{e,g}|} = \frac{\mu_q(e_{(g)})}{\Phi_q(e_{(g)})} \cdot \frac{\Phi_q(f)}{|f|} ,
\end{multline*}
where $f_{e,g}$ is the highest degree factor of $f$ that is relatively prime to $e_{(g)}$. We plug the latter into Eq.~\eqref{eq:I3} and obtain
\[ \I_{f,g}(\alpha) = \frac{\Phi_q(f)}{|fg|} \sum_{e\mid fg} \frac{\mu_q(e_{(g)})}{\Phi_q(e_{(g)})} B_\alpha(e) . \]
The result follows upon replacing $B_\alpha(e)$ by $\sum_{\oord(\psi)=e}\psi(\alpha)$.
\end{proof}

\section{On $(f, g)$-freeness through polynomial values} \label{sec:freepolval}

For polynomials $h, H\in \F_{q^n}[x]$, we intend to study the number of pairs $(h(y), H(y))$ with $y\in \F_{q^n}$ such that $h(y)$ is $(f, g)$-free and $H(y)$ is $(F, G)$-free. We aim to employ Theorem~\ref{thm:weil} but in order to effectively use this result, we must
restrict ourselves to those pairs $(h,H)$, such that for every $a,b\in\F_{q^n}$ with $(a,b)\neq (0,0)$, the polynomial $ah(x) + bH(x)$ is nonsingular. For convenience, we will call such pairs $(h,H)$ \emph{nonsingular}, while if there exist some $a,b\in\F_{q^n}$, with $(a,b)\neq (0,0)$, such that $ah(x) + bH(x)$ is singular, we will call the pair $(h,H)$ \emph{singular}.

We give a simple example, illustrating that the above can be necessary in order to have at least one pair $(h(y), H(y))$ of polynomial values with prescribed freeness, thus it is natural to restrain ourselves to nonsingular pairs.

\begin{example}
Write $q=p^k$ and suppose that $H(x)=h(x)^q$. In particular, $H(x)-h(x)=h(x)^q-h(x)=r(x)^p-r(x)$ with $r(x)=\sum_{i=0}^{k-1}h(x)^{p^{i}}$, thus $(h,H)$ is singular.
For any $y\in \fqn$, we have that $H(y)=x\circ h(y)$.  By Lemma~\ref{lem:gcd}, the $\F_q$-orders of $h(y), H(y)$ coincide. From Remark~\ref{remark:freeness}, we can produce many examples in which no pair $(h(y), H(y))$ with $y\in \fqn$ satisfies that $h(y)$ is $(f, g)$-free and $H(y)$ is $(F, G)$-free.
\end{example}
The next example provides a large family of nonsingular pairs.
\begin{example}\label{ex:main}
Observe that if $h, H\in \F_q[x]$ are nonzero polynomials of degree not divisible by $p$, then $ah(x)+bH(x)$ is nonsingular 
for $(a, b)\ne (0, 0)$ unless $\deg(h)=\deg(H)=1$. In particular, if $\gcd(\deg(h)\cdot\deg(H), p)=1$ and $\deg(h)\cdot \deg(H)>1$, the pair $(h, H)$ is nonsingular.
\end{example}


Regarding singularity, in the proceeding section, see Proposition~\ref{prop:nonsingular}, we highlight another case where singular pairs require special treatment. However, for now, we confine ourselves to nonsingular pairs and we obtain the following result. 

\begin{theorem}\label{thm:main}
Fix $q$ a prime power and $n\ge 1$ a positive integer. Let $f, F\in \F_q[x]$ be divisors of $x^n-1$ and let $g, G\in \F_q[x]$ be such that $g$ divides $\frac{x^n-1}{f(x)}$ and $G$ divides $\frac{x^n-1}{F(x)}$. Let $h, H\in \F_{q^n}[x]$ be such that 
$(h,H)$ is nonsingular.
Set $D_1:=\deg (fFgG)$ and let $D_2+1$ be the maximum degree of the polynomials $ah(x)+bH(x)$ as $a, b$ run over the elements of $\F_{q^n}$ with $(a, b)\ne (0, 0)$.
Then the number $N_{h, H}=N_{h, H}(f, g, F, G)$ of elements $y\in \F_{q^n}$ such that $h(y)$ is $(f, g)$-free and $H(y)$ is $(F, G)$-free satisfies
$$\frac{N_{h, H}\cdot q^{D_1}}{\Phi_q(f)\Phi_q(F)}=q^n+\ell,$$
where $|\ell|\le D_2W(f)W(F)q^{\deg(gG)+n/2}$.
\end{theorem}

\begin{proof}
It follows by the definition that \[ N_{h, H}=\sum_{w\in \F_{q^n}}\I_{f, g}(h(w))\cdot \I_{F, G}(H(w)). \] From Proposition~\ref{prop:vinogradov}, for $\delta=\frac{\Phi_q(f)\Phi_q(F)}{q^{D_1}}$, we have that 
\[ 
\frac{N_{h, H}}{\delta} 
= \sum_{t|fg, \; T|FG}\frac{\mu_q(t_{(g)})\cdot \mu_q(T_{(G)})}{\Phi_q(t_{(g)})\cdot \Phi_q(T_{(G)})}\sum_{\substack{\oord(\psi)=t\\ \oord(\psi')=T}}G_{h, H}(\psi, \psi'),
\] 
where $G_{h, H}(\psi, \psi')=\sum_{w\in \F_{q^n}}\psi(h(w))\psi'(H(w))$. Fix $t\mid fg$ and $T\mid FG$ and let $\psi, \psi'$ be additive characters of $\fqn$ with $\F_q$-orders $t$ and $T$, respectively. If $(t, T)\ne (1, 1)$, we have that $(\psi, \psi')=(\psi_a, \psi_b)$ for some $a, b\in \fqn$ with $(a, b)\ne (0, 0)$. In this case, $\psi(h(x))\cdot \psi'(H(x))=\psi_1(ah(x)+bH(x))$ and $(a, b)\ne (0, 0)$, where $\psi_1$ is the canonical additive character (hence nontrivial).
Since $(h,H)$ is nonsingular, 
$ah(x)+bH(x)$ is nonsingular 
and so Theorem~\ref{thm:weil} yields $|G_{h, H}(\psi, \psi')|\le D_2q^{n/2}$. For $(t, T)=(1, 1)$, $\psi=\psi'=\psi_0$ is the trivial additive character of $\fqn$ and so $|G_{h, H}(\psi_0, \psi_0)|=q^n$. 
Applying the above estimates we obtain
\[ \left|\frac{N_{h, H}}{\delta}- q^n\right|\le D_2q^{n/2}\cdot M, \]
where 
\[ M=\sum_{\substack{t\mid fg,\;  T\mid FG\\ (t, T)\ne (1, 1)}}\frac{|\mu_q(t_{(g)})\cdot \mu_q(T_{(G)})|}{\Phi_q(t_{(g)})\cdot \Phi_q(T_{(G)})}\sum_{\substack{ \oord(\psi)=t\\ \oord(\psi')=T}}1=T(fg, g)\cdot T(FG, G)-1, \]
and $T(f,g)$ is as in Lemma~\ref{lem:arithfun}. Note that, in the last equality we used the fact that for each monic divisor $R\in \F_q[x]$ of $x^n-1$ there exist $\Phi_q(R)$ characters of $\fqn$ of $\F_q$-order $R$. 
Further, Lemma~\ref{lem:arithfun} implies $T(fg, g)=q^{\deg(g)}\cdot W(f)$ and $T(FG,G) = q^{\deg(G)} W(F)$, so 
\begin{align*} \left|\frac{N_{f,F}}{\delta}-q^n\right| & \le  D_2q^{n/2}(W(f)W(F)q^{\deg(gG)}-1) \\  & \le D_2W(f)W(F)q^{\deg(gG)+n/2}. \qedhere \end{align*}
\end{proof}
%
The corollary below is an immediate consequence of Theorem~\ref{thm:main} and it provides us with practical sufficient condition for the existence of elements $y\in\F_{q^n}$, such that $h(y)$ is $(f,g)$-free and $H(y)$ is $(F,G)$-free.
\begin{corollary}\label{cor:main}
Assume the notation and the hypotheses of Theorem~\ref{thm:main}. We have that $N_{h,H}>0$ if \[ q^{n/2-\deg(gG)} > D_2 W(f)W(F). \]
\end{corollary}
Here, we point out that one of the main benefits of the introduction of $(r,n)$-freeness in \cite{cohenkapetanakisreis22} was its natural and seamless compatibility with the Cohen-Huczynska prime sieve \cite{cohenhuczynska03}. As we will see in the following result, this benefit is carried over to the additive analogue that we study here and it enables us to further weaken the condition of Corollary~\ref{cor:main}. Also, since the two proofs are very similar, we leave the proof as an exercise to the interested reader and refer them to Proposition~19 and Theorem~20 of \cite{cohenkapetanakisreis22} for the multiplicative analogue.
\begin{theorem} \label{thm:sieve}
Fix $q$ a prime power and $n\ge 1$ a positive integer. Let $f, F\in \F_q[x]$ be squarefree divisors of $x^n-1$ and let $g, G\in \F_q[x]$ be such that $g$ divides $\frac{x^n-1}{f(x)}$ and $G$ divides $\frac{x^n-1}{F(x)}$. Let $h, H\in \F_{q^n}[x]$ be such that 
$(h,H)$ is nonsingular.
Additionally, write $f=kp_1\cdots p_u$ and $F=KP_1\cdots P_v$, where $p_1,\ldots ,p_u,P_1,\ldots ,p_v$ are irreducible polynomials, such that \[ \delta := 1 - \sum_{i=1}^u 1/|p_i| - \sum_{j=1}^v 1/|P_j| >0. \]
Finally, set $D_1:=\deg (kKgG)$ and let $D_2$ be the maximum degree of the polynomials $ah(x)+bH(x)$ as $a, b$ run over the elements of $\F_{q^n}$ with $(a, b)\ne (0, 0)$.
Then the number $N_{h, H}=N_{h, H}(f, g, F, G)$ of elements $y\in \F_{q^n}$ such that $h(y)$ is $(f, g)$-free and $H(y)$ is $(F, G)$-free satisfies
\[ \frac{N_{h, H}\cdot q^{D_1}}{\delta\cdot \Phi_q(k)\Phi_q(K)}=q^n+\ell, \]
where $|\ell|\le D_2W(k)W(K)\left( \frac{u+v}{\delta} + 2 \right) q^{\deg(gG)+n/2}$.
\end{theorem}
The following corollary translates the above theorem into a practical sufficient condition for the existence of elements with the desired properties, in a similar fashion as Corollary~\ref{cor:main} did for Theorem~\ref{thm:main}.
\begin{corollary} \label{cor:sieve}
Assume the notation and the assumptions of Theorem~\ref{thm:sieve}. Then $N_{h, H}>0$, given that
\[ q^{n/2-\deg(gG)} > D_2W(k)W(K)\left( \frac{u+v}{\delta} + 2 \right) . \]
\end{corollary}
%
%
%
%
\section{Normal points on Artin-Schreier curves} \label{sec:ascurves}
Throughout this section, we consider $\F_p$ as the base field, i.e., we write $q=p^n$ and so $\F_q=\F_{p^n}$ is viewed as the $n$-degree extension of $\F_p$. In particular, we adopt the concepts and results from Sections~\ref{sec:preparation}, \ref{sec:freeness} and \ref{sec:freepolval} with $q=p$.

Recall that in \cite{cohenkapetanakisreis22}, the authors explored the existence of points $(x_0, y_0)\in \F_{q}\times \F_q$ on curves $y^n=f(x)$ such that $x_0$ and $y_0$ are primitive elements of $\F_q$. As earlier mentioned this was done through a generalized concept of freeness over the cyclic group $\F_q^*$, which is just the multiplicative analogue of what we developed in Sections~\ref{sec:preparation} and \ref{sec:freeness}. 
Here we explore the natural additive counterpart of this problem. Namely, we study the existence of $\F_q$-affine points on Artin-Schreier curves whose coordinates are normal over $\F_p$.
Recall that an \emph{Artin-Schreier curve} is a plane curve defined over $\overline{\F}_q$ by an affine equation $y^p-y=f(x)$, where $p$ is the characteristic of $\F_q$ and $f\in \F_q[x]$ is a polynomial not of the form $r(x)^p-r(x)$ for some $r\in \overline{\F}_q[x]$, or, equivalently, $f$ is nonsingular over every extension of $\F_q$. In fact, the definition includes rational functions $f\in \F_q(x)$ but we are going to consider only polynomials. We introduce the following definition.
 
\begin{definition}
Given an Artin-Schreier curve $\mathfrak A_f: y^p-y=f(x)$, an $\F_q$-rational point $(x_0, y_0)\in \F_q\times \F_q$ is an {\em $\F_q$-normal point} of $\mathfrak A_f$ if both $x_0$ and $y_0$ are normal over the prime field $\F_p$. 
\end{definition}

The following lemma relates the existence of normal points on Artin-Schreier curves to the existence of special pairs of elements in $\F_q$ with prescribed freeness. 

\begin{lemma}\label{lem:connect}
An Artin-Schreier curve $\mathfrak A_f: y^p-y=f(x)$ admits an $\F_q$-normal point if and only if there exists an element $z\in \F_q$ that is normal over $\F_p$ such that  
$f(z)$ has $\F_p$-order $\frac{x^n-1}{x-1}$. The latter is equivalent to the existence of an element $z\in \F_q$ such that $z$ is $(x^n-1, 1)$-free and $f(z)$ is $\left(\frac{x^n-1}{x-1}, x-1\right)$-free.  
\end{lemma}
\begin{proof}
The second statement follows directly by Remark~\ref{remark:freeness}. For the first statement, assume that $\mathfrak A_f$ admits an $\F_q$-normal point $(x_0, y_0)$. Hence $z=x_0$ is normal over $\F_p$ and $f(z)=f(x_0)=y_0^p-y_0$. Since $y_0$ is also normal over $\F_p$, Corollary~\ref{cor:ord} implies that $f(z)=(x-1)\circ y_0$ has $\F_p$-order $\frac{x^n-1}{x-1}$.
Conversely, suppose that there exists an element $z\in \F_q$ that is normal over $\F_p$ such that $f(z)$ has $\F_p$-order $\frac{x^n-1}{x-1}$. From Corollary~\ref{cor:ord}, we have that 
$f(z)=(x-1)\circ ay_1$, where $a\in \F_p^*$ and $y_1$ is normal over $\F_p$. It is clear that $ay_1$ is also normal over $\F_p$. In particular, the point $(x_0, y_0)=(z, ay_1)$ has normal coordinates and belongs to the curve $\mathfrak A_f$, i.e., the curve $\mathfrak A_f$ admits an $\F_q$-normal point. 
\end{proof}

We obtain the following result.

\begin{corollary}\label{cor:ineq}
Assume the notation of Theorem~\ref{thm:main}. An Artin-Schreier curve $\mathfrak A_f: y^p-y=f(x)$ admits an $\F_q$-normal point whenever $$N_{x, f}\left(x^n-1, 1, \frac{x^n-1}{x-1}, x-1\right)>0.$$ In particular, the latter holds if the pair $(x, f(x))$ is not singular and
$$p^{\frac{n}{2}-1}\ge (\deg f-1)W(x^n-1)W\left(\frac{x^n-1}{x-1}\right),$$
where for $g\in \F_p[x]$, $W(g)$ denotes the number of distinct monic squarefree divisors of $g$, defined over $\F_p$.
\end{corollary}
\begin{proof}
The first statement follows directly by Lemma~\ref{lem:connect}. For the second statement, observe that our assumption on the pair $(x, f(x))$ allows us to employ Theorem~\ref{thm:main} and then the result follows 
from Corollary~\ref{cor:main}.
\end{proof}
In Example~\ref{ex:main} we comment that, if $f(x)$ is linear, then the pair $(x, f(x))$ is singular. In fact, for $f(x)=cx+d$ with $c\in \F_{q}^*$, we have that $ax+bf(x)=r(x)^p-r(x)-d$ for $(a, b)=(c, -1)\ne (0, 0)$ and $r(x)=0$. Motivated by Corollary~\ref{cor:ineq}, in the following proposition we show that we can actually obtain some existence results when $f(x)$ has degree one, without going through the character sum method.

\begin{proposition}\label{prop:nonsingular}
Let $f(x)=ax+b\in \F_p[x]$, where $a\ne 0$ and $q=p^n$. If $b=0$ or $n\equiv 0\pmod p$, then the Artin-Schreier curve $\mathfrak A_f: y^p-y=f(x)$ does not admit an $\F_q$-normal point. On the other hand, if $n\not\equiv 0\pmod p$, for every element $\alpha\in \F_q$ that is normal over $\F_p$, there exists a linear polynomial $F(x)=x+b\in \F_p[x]$ with $b\ne 0$ such that its associated Artin-Schreier curve $\mathfrak A_F: y^p-y=F(x)$ contains an $\F_q$-normal point of the form $(\alpha, y_0)$.  
\end{proposition}
\begin{proof}
Suppose, by contradiction that there exists an $\F_q$-normal point $(\alpha, \beta)$ of $\mathfrak A_f$. In particular, for the trace polynomial $T_n(x):=\sum_{i=0}^{n-1}x^{p^{i-1}}$, we have that 
$$0=\beta^q-\beta=T_n(\beta^p-\beta)=T_n(f(\alpha))=T_n(a\alpha+b)=aT_n(\alpha)+nb=aT_n(\alpha),$$
where in the last equality we used the fact that $b=0$ or $n\equiv 0\pmod p$. Since $a\ne 0$, we obtain that $T_n(\alpha)=0$. We observe that $T_n(\alpha)=\frac{x^n-1}{x-1}\circ \alpha$ and this is a contradiction with the assumption that $\alpha$ is normal over $\F_p$.

Now let $\alpha\in \F_q$ be normal over $\F_p$. In particular, $\delta:=T_n(\alpha)=\frac{x^n-1}{x-1}\circ \alpha\ne 0$. Since $n\not\equiv 0\pmod p$, there exists $\Delta\in \F_p^*$ such that $n\Delta=1\in \F_p$. In this case, for $F(x)=x-\delta\cdot \Delta$, we have that 
$$T_n(F(\alpha))=T_n(\alpha-\delta\cdot \Delta)=T_n(\alpha)-\delta\cdot \Delta\cdot n=\delta-\delta=0.$$
From Theorem~3.78 and Corollary~3.79 in \cite{lidlniederreiter97}, there exists $\beta\in \F_q$ such that $F(\alpha)=\beta^p-\beta$.
 In particular, $(\alpha, \beta+t)$ is an $\F_q$-rational point of $\mathfrak A_F$ for arbitrary $t\in \F_p$. It remains to prove
that there exists some $t_0\in \F_p$ such that $\beta+t_0$ is normal over $\F_p$, hence producing the $\F_p$-normal point $(\alpha, \beta+t_0)$.

We first prove that $\oord(F(\alpha))=\frac{x^n-1}{x-1}$. As $\frac{x^n-1}{x-1}\circ \alpha=T_n(F(\alpha))=0$, it follows that $\oord(F(\alpha))$ divides $\frac{x^n-1}{x-1}$. On the other hand, since $\oord(F(\alpha))\circ F(\alpha)=0$, we have that \begin{align*}0 =((x-1)\cdot\oord(F(\alpha)))\circ F(\alpha)  &=\oord(F(\alpha))\circ (F(\alpha)^p-F(\alpha)) \\ &=\oord(F(\alpha))\circ (\alpha^p-\alpha) \\ &=((x-1)\cdot \oord(F(\alpha)))\circ \alpha.\end{align*}
Hence $\oord(\alpha)=x^n-1$ divides $(x-1)\oord(F(\alpha))$. Therefore,  $\oord(F(\alpha))$ is divisible by $\frac{x^n-1}{x-1}$ and so $\oord(F(\alpha))=\frac{x^n-1}{x-1}$. Now, since $F(\alpha)=(x-1)\circ (\beta+t)$ for every $t\in \F_p$, it follows by Lemma~\ref{lem:gcd} that $\oord(\beta+t)=x^n-1$ or $\frac{x^n-1}{x-1}$. In particular, $\oord(\beta+t)=x^n-1$ if and only if $T_n(\beta+t)=\frac{x^n-1}{x-1}\circ (\beta+t)\ne 0$. Since $n\Delta=1\in \F_p$, it follows that $T_n(\beta+t_0)=T_n(\beta)+nt_0\ne 0$ for every $t_0\in \F_p$ with $t_0\ne  -\Delta\cdot T_n(\beta)$. In particular, for any such $t_0$, we have $\oord(\beta+t_0)=x^n-1$ and so $\beta+t_0$ is normal over $\F_p$.
\end{proof}

\subsection{Proof of Theorem~\ref{thm:as}}
We observe that if  $1<\deg(f)\le p+1$, then the pair $(x, f(x))$ is not singular whenever $f(x)$ is not of the form $ax^p+bx+c$ with $a, b, c\in \F_q$. Moreover, assuming the latter, from our previous discussion we conclude that the polynomial $y^p-y-f(x)$ gives rise to an Artin-Schreier curve. From Corollary~\ref{cor:ineq} and the inequality $\deg(f)-1\le p$, it suffices to verify the following inequality
\begin{equation}\label{eq:main}p^{\frac{n}{2}-2}\ge W(x^n-1)W\left(\frac{x^n-1}{x-1}\right).\end{equation}

We start with the following technical lemma: for its proof, see Lemma 3.7 in~\cite{jv}.

\begin{lemma}
Let $p$ be a prime and let $n\ge 2$ be a positive integer. Then $W(x^n-1)\le 2^{\frac{n+a}{b}}$, where $(a, b)=(14, 5), (20, 4)$ and $(18, 3)$ for $p=2, 3$ and $p=5$, respectively. Moreover, for $7\le p\le 23$ and $p\ge 29$ we can take $(a, b)=(p-1, 2)$ and $(a, b)=(0, 1)$, respectively.  
\end{lemma}

Upon combining the previous lemma with the trivial bound \[ W\left(\frac{x^n-1}{x-1}\right)\le \min \{ W(x^n-1),2^{n-1}\}, \] we obtain that Ineq.~\eqref{eq:main} holds whether $n>4$ and $p>290000$, while for smaller values of $p$ we also restrict, in Table~\ref{tab:exc-1}, the possible pairs $(n,p)$ where Ineq.~\eqref{eq:main} might not hold.

\begin{table}[h]
\begin{center}\scriptsize
\begin{tabular}{|c|c|}
\hline $p$ & $n$\\  \hline
2 &  $\le 75$\\ \hline
3 &    $\le 45$\\ \hline
5 & $\le 33$\\ \hline
7 &  $\le 28$\\ \hline
11, 13, 17, 19, 23 &  $\le 24$ \\ \hline
$\le 29$ & $\le 20$\\ \hline
$\le 100$ & $\le 9$\\ \hline
$\le 200 $ & $\le 7$\\ \hline
$\le 500$ & $\le 6$\\ \hline
$\le 2100$ & $\le 5$\\ \hline
\end{tabular}
\end{center}
\caption{Pairs $(n,p)$ that may not satisfy Ineq.~\eqref{eq:main}.}
\label{tab:exc-1}
\end{table}

Then, we consider the (finite) set of pairs $(n,p)$ included in Table~\ref{tab:exc-1} and directly verify Ineq.~\eqref{eq:main}. In other words, we explicitly compute the value of $W(x^n-1)$ and $W(\frac{x^n-1}{x-1})$. For $n=5$, there are 5779 primes that fail to satisfy Ineq.~\eqref{eq:main}, while the explicit list of exceptional pairs $(n,p)$, with $n\geq 6$, is presented in Table~\ref{tab:exc-2}.

\begin{table}[h!]
\begin{center}\scriptsize
\begin{tabular}{|c|p{0.8\linewidth}|c|}
\hline $n$ & $p$ & $\#$\\  \hline\hline
 6 & 2,  3,  5,  7,  11,  13,  17,  19,  23,  29,  31,  37,  41,  43,  47,  53,  59,  61,  67,  71,  73, 79,  83,  89,  97,  101,  103,  107,  109,  113,  127,  139,  151,  157,  163,  181,  193,  199,  211, 223,  229,  241,  271,  277,  283,  307,  313,  331,  337,  349,  367,  373,  379,  397,  409,  421,  433, 439,  457,  463,  487,  499,  523,  541,  547,  571,  577,  601,  607,  613,  619,  631,  643,  661,  673, 691,  709,  727,  733,  739,  751,  757,  769,  787,  811,  823,  829,  853,  859,  877,  883,  907,  919, 937,  967,  991,  997,  1009,  1021,  1033,  1039,  1051,  1063,  1069,  1087,  1093,  1117,  1123,  1129, 1153,  1171,  1201,  1213,  1231,  1237,  1249,  1279,  1291,  1297,  1303,  1321,  1327,  1381,  1399, 1423,  1429,  1447,  1453,  1459,  1471,  1483,  1489,  1531,  1543,  1549,  1567,  1579,  1597,  1609, 1621,  1627,  1657,  1663,  1669,  1693,  1699,  1723,  1741,  1747,  1753,  1759,  1777,  1783,  1789, 1801,  1831,  1861,  1867,  1873,  1879,  1933,  1951,  1987,  1993,  1999,  2011,  2017,  2029 & 168 \\ \hline
7 & 2, 3, 13, 29, 43, 71, 113, 127, 197, 211, 239, 281, 337, 379 & 14 \\ \hline
8 & 3, 5, 7, 11, 13, 17, 19, 29, 37, 41, 73, 89, 97, 113, 137 & 15 \\ \hline
9 &2, 7, 19, 37, 73 & 5 \\ \hline
10 & 2, 3, 11, 31, 41, 61, 71 & 7 \\ \hline
11 & 23 & 1 \\ \hline
12 & 5, 7, 13, 19 & 4 \\ \hline
13 & 3 & 1 \\ \hline
14 & 2, 29 & 2 \\ \hline
15 & 2 & 1 \\ \hline
16 & 3, 5, 7, 17 & 4 \\ \hline
18 & 19 & 1 \\ \hline
20 & 3, 11 & 2 \\ \hline
21 & 2 & 1 \\ \hline
22 & 23 & 1 \\ \hline
24 & 5, 7 & 2 \\ \hline
26 & 3 & 1 \\ \hline
\multicolumn{2}{|r|}{\textbf{Total:}} & 230 \\ \hline
\end{tabular}
\end{center}
\caption{Pairs $(n,p)$ that do not satisfy Ineq.~\eqref{eq:main}, where $n\ge 6$.}
\label{tab:exc-2}
\end{table}

Finally, we turn our attention to the potential of ruling out most of the exceptional pairs using the sieve, as described in Theorem~\ref{thm:sieve} and Corollary~\ref{cor:sieve}. More precisely, the aforementioned results imply that the condition of Corollary~\ref{cor:ineq} may be improved as
\begin{equation}\label{eq:sieve}
p^{\frac n2 - 1} \geq (\deg f-1) W(k_1)W(k_2) \left( \frac{2u}{\delta} + 2 \right) ,
\end{equation}
where (the squarefree part) of $(x^n-1)/(x-1)$ is $k_1p_1\cdots p_u$ and the (squarefree part) of $x^n-1$ is $k_2 r_1\cdots r_v$ for some irreducible polynomials $p_1,\ldots ,p_u,r_1,\ldots ,r_v$, such that 
\[  \delta := 1 - \sum_{i=1}^u 1/|p_i| - \sum_{j=1}^v 1/|r_j| \]
is positive. In particular, in our test, for each pair $(n,p)$, we choose the polynomials $p_1,\ldots ,p_u,r_1,\ldots ,r_v$ in such a way that $u$ and $v$ are maximum and $\delta$ remains positive and check whether Ineq.~\eqref{eq:sieve} holds. A quick computer test reveals that among the aforementioned 6009 pairs $(n,p)$ that did not satisfy Ineq.~\eqref{eq:main}, just 5 prove to be persistent enough to fail this test as well. These pairs $(n,p)$ are $(5,2), (5,5), (6,2), (6,3)$ and $(6,7)$. 
%
This concludes the proof of Theorem~\ref{thm:as}.


\section{Acknowledgments}

We are grateful to the anonymous reviewer for their efforts in reviewing our manuscript and their suggestions and improvements.

\section{Disclosure statement}

The authors do not work for, consult, own shares in or receive funding from any
company or organization that would benefit from this article, and have disclosed
no relevant affiliations beyond their academic appointment.


\begin{thebibliography}{99}
\bibitem{jv} J.~J.~R.~Aguirre, V.~G.~L.~Neumann.
\newblock Existence of primitive $2$-normal elements in finite fields.
\newblock{\em Finite Fields Appl.}, 73: 101864, 2021.

\bibitem{bailey19} G.~Bailey, S.~D. Cohen, N.~Sutherland, T.~Trudgian.
\newblock Existence results for primitive elements in cubic and quartic extensions of a finite field.
\newblock {\em Math. Comp.}, 88(316):931--947, 2019.

\bibitem{BCLT} A.~R.~Booker, S.~D.~Cohen, N.~Leong and T.~Trudgian.
\newblock Primitive element pairs with a prescribed trace in the cubic extension of a finite field.
\newblock {\em Bull. Aust. Math. Soc.}, 106(3):458--462, 2022.

  \bibitem{cohenhuczynska03} S.~D.~Cohen and S.~Huczynska.
  \newblock The primitive normal basis theorem {{--}} without a computer.
\newblock {\em J. London Math. Soc.}, 67(1):41--56, 2003.

  \bibitem{cohenkapetanakisreis22} S.~D.~Cohen, G.~Kapetanakis and L.~Reis.
\newblock The existence of $\F_q$‐primitive points on curves using freeness.
\newblock {\em Comptes Rendus Math.}, 360(G6):641--652, 2022.

\bibitem{dh}
W.~Diffie, M.~ Hellman.
\newblock New directions in cryptography.
\newblock {\em IEEE Trans. Information Theory}, 22(6):644--654, 1976.

\bibitem{GAO}
{S. Gao},
{\em Normal basis over finite fields}, (PhD thesis, University of Waterloo, 1993).

\bibitem{HMPT} S.~Huczynska, G.~L.~Mullen, D.~Panario and D.~Thomson,
\newblock Existence and properties of $k$-normal elements over finite fields.
\newblock {\em Finite Fields Appl.}, 24:170--183, 2013.

\bibitem{kapetanakisreis18} G.~Kapetanakis and L.~Reis,
\newblock Variations of the primitive normal basis theorem.
\newblock {\em Des. Codes Cryptogr.}, 87(7):1459--1480, 2019.

 \bibitem{lenstraschoof87} H.~W.~Lenstra Jr and R.~J.~Schoof.
 \newblock Primitive normal bases for finite fields.
\newblock {\em Math. Comp.}, 48(177):217--231, 1987.

  \bibitem{lidlniederreiter97} R.~Lidl and H.~Niederreiter.
\newblock {\em Finite Fields}, vol.~20 of \emph{Encycl. Math. Appl.}, 
\newblock Cambridge University Press, Cambridge, second edition, 1997.


\bibitem{R20} L.~Reis,
\newblock Counting solutions of special linear equations over finite fields. 
\newblock{\em Finite Fields Appl.}, 68: 101759, 2020.

\end{thebibliography}
\end{document}